\documentclass[12pt]{amsart}
\usepackage[margin=1.2in]{geometry}
\usepackage{graphicx,latexsym}
\usepackage{comment}
\usepackage{tikz}
\usepackage{amsfonts, amssymb, amsmath, amsthm, bm, hyperref}
\usepackage{tikz}
\usetikzlibrary{matrix,arrows}

\usepackage{color}

\newcommand{\N}{\mathbb{N}}
\newcommand{\Z}{\mathbb{Z}}
\newcommand{\R}{\mathbb{R}}
\newcommand{\C}{\mathbb{C}}
\newcommand{\HH}{\mathbb{H}}
\newcommand{\SMA}{\mathcal{S}^{\A}_{\M}}
\newcommand{\DM}{\mathcal{D}^{{\M}}}

\newcommand{\SAM}{\mathcal{S}^{\M}_{\A}}
\newcommand{\CM}{C_{\M}[0,\infty)}

\newcommand{\M}{{\bf{M}}}
\newcommand{\NN}{{\bf{N}}}
\newcommand{\A}{{\bf{A}}}

\newcommand{\dx}{{\rm d}x }

\newcommand{\dt}{{\rm d}t }
\newcommand{\dxi}{{\rm d}\xi }

\newcommand{\supp}{\operatorname{supp}}
\newcommand{\lc}{\operatorname{(lc)}}
\newcommand{\wlc}{\operatorname{(wlc)}}
\newcommand{\dc}{\operatorname{(dc)}}
\newcommand{\mg}{\operatorname{(mg)}}
\newcommand{\nq}{\operatorname{(nq)}}
\newcommand{\snq}{\operatorname{(snq)}}

\newcommand{\om}{\omega}
\newcommand{\ga}{\gamma}

\newtheorem{theorem}{Theorem}[section]
\newtheorem{proposition}[theorem]{Proposition}
\newtheorem{lemma}[theorem]{Lemma}
\newtheorem{corollary}[theorem]{Corollary}

\theoremstyle{definition}

\newtheorem{example}[theorem]{Example}

\theoremstyle{remark}
\newtheorem{remark}[theorem]{Remark}

\numberwithin{equation}{section}

\makeatletter
\DeclareRobustCommand\widecheck[1]{{\mathpalette\@widecheck{#1}}}
\def\@widecheck#1#2{%
    \setbox\z@\hbox{\m@th$#1#2$}%
    \setbox\tw@\hbox{\m@th$#1%
       \widehat{%
          \vrule\@width\z@\@height\ht\z@
          \vrule\@height\z@\@width\wd\z@}$}%
    \dp\tw@-\ht\z@
    \@tempdima\ht\z@ \advance\@tempdima2\ht\tw@ \divide\@tempdima\thr@@
    \setbox\tw@\hbox{%
       \raise\@tempdima\hbox{\scalebox{1}[-1]{\lower\@tempdima\box
\tw@}}}%
    {\ooalign{\box\tw@ \cr \box\z@}}}
\makeatother

\begin{document}

\title[Stieltjes moment mapping in Gelfand-Shilov spaces]{Injectivity and surjectivity of the Stieltjes\\
moment mapping in Gelfand-Shilov spaces}

\author{Andreas Debrouwere}
\address{Department of Mathematics, Louisiana State University, United States}
\email{adebrouwere@lsu.edu}

\author{Javier Jim\'enez-Garrido}
\address{Instituto de Investigaci\'on en Matem\'aticas IMUVA\newline \indent Universidad de Valladolid, Valla\-dolid, Spain}
\email{jjjimenez@am.uva.es}

\author{Javier Sanz}
\address{Departmento de \'Algebra, An\'alisis Matem\'atico, Geometr{\'\i}a y Topolog{\'\i}a\newline
\indent Instituto de Investigaci\'on en Matem\'aticas IMUVA\newline
\indent Universidad de Valladolid, Valladolid, Spain}
\email{jsanzg@am.uva.es}

\begin{abstract}

The Stieltjes moment problem is studied in the framework of general Gelfand-Shilov spaces defined via weight sequences. We characterize the injectivity and surjectivity of the Stieltjes moment mapping, sending a function to its sequence of moments, in terms of growth conditions for the defining weight sequence. Finally, a related moment problem at the origin is  studied.

\end{abstract}

\maketitle

\section{Introduction}
The moment problem, with its many variations and generalizations, has a long and rich tradition that goes back to the seminal work of Stieltjes \cite{Stieltjes}. In 1939, Boas \cite{Boas} and P\'olya \cite{Polya} independently showed that, for every sequence $(c_p)_{p=0}^\infty$ of complex numbers,
there is a function $F$ of bounded variation such that
$$
\int_0^\infty x^p {{\rm d}F(x) }=c_p,\qquad p\in\N=\{0,1,2,\ldots\}.
$$
This result was greatly improved by A.~J.~Dur\'an \cite{Duran} in 1989 who proved, in a constructive way,  that, for every sequence $(c_p)_{p \in \N}$ of complex numbers, the infinite system of linear equations
\begin{equation}
\int_0^\infty x^p \varphi(x) \dx =c_p,\qquad p \in \N,
\label{moment-problem}
\end{equation}
admits a solution $\varphi \in \mathcal{S}(0,\infty)$, that is,  $\varphi$ belongs to the Schwartz space $\mathcal{S}(\R)$ of rapidly decreasing smooth functions and $\operatorname{supp} \varphi \subseteq [0,\infty)$. We would like to point out  that this result also follows from a short non-constructive argument via Eidelheit's theorem \cite[Thm.\ 26.27]{M-V}.

In this article, we study the (unrestricted) Stieltjes moment problem \eqref{moment-problem} in the context of Gelfand-Shilov spaces defined via weight sequences \cite{G-S}; see \cite{ChungKimYeom,L-S08,L-S09} for earlier works in this direction. Namely,  given two sequences of positive real numbers $\M=(M_p)_{p\in\N}$ and $\A=(A_p)_{p\in\N}$, we consider the spaces
$\mathcal{S}_{\M}^{\A}(0,\infty)$ and $\mathcal{S}_{\M}(0,\infty)$
consisting of all $\varphi \in \mathcal{S}(0,\infty)$ such that there exists $h>0$ with
$$
\sup_{p,q \in \N} \sup_{x \in \R} \frac{|x^p\varphi^{(q)}(x)|}{h^{p+q}p!M_pq!A_q} < \infty
$$
and
$$
\sup_{p\in \N} \sup_{x \in \R} \frac{|x^p\varphi^{(q)}(x)|}{h^{p}p!M_p} < \infty, \qquad \forall q \in \N,
$$
respectively. Obviously, $\mathcal{S}_{\M}^{\A}(0,\infty)\subset\mathcal{S}_{\M}(0,\infty)$. Now suppose that $\M$ is derivation closed, that is,  $M_{p+1} \leq C_0H^{p+1}M_p$, $p \in \N$, for some $C_0,H \geq 1$. Clearly, for every $\varphi \in \mathcal{S}_{\M}(0,\infty)$, its sequence of moments $(\mu_p(\varphi))_{p\in\N} = (\int_0^\infty x^p \varphi(x) \dx)_{p \in \N}$ belongs to the sequence space
$
\Lambda_{\M} = \{ (c_p)_{p \in \N} \, | \, \exists h > 0 \, : \,   \sup_{p \in \N} \frac{|c_p|}{h^p p!M_p} < \infty  \}.
$
It is then natural to ask about the surjectivity of the Stieltjes moment mapping $\mathcal{M}$, given by $\mathcal{M}(\varphi)=(\mu_p(\varphi))_{p\in\N}$, when defined on either $\mathcal{S}_{\M}^{\A}(0,\infty)$ or $\mathcal{S}_{\M}(0,\infty)$ and with range $\Lambda_{\M}$.
As a first result, following a technique of A.~L.~Dur\'an and R.~Estrada~\cite{D-E} that combines the Fourier transform with the
Borel-Ritt theorem from asymptotic analysis, S.-Y.~Chung, D.~Kim and Y.~Yeom~\cite[Thm.\ 3.1]{ChungKimYeom} proved the surjectivity of $\mathcal{M}$ for $\mathcal{S}_{\M}(0,\infty)$ with $\M=(p!^{\alpha-1})_{p\in\N}$ (the Gevrey sequence) whenever $\alpha>2$. Subsequently, A. Lastra and the third author~\cite{L-S08} refined this result by obtaining linear continuous right inverses for the Stieltjes moment mapping between suitable Fr\'echet subspaces of $\mathcal{S}_{p!^{\alpha-1}}(0,\infty)$ and extended this result \cite{L-S09} to $\mathcal{S}_{\M}(0,\infty)$ for general strongly regular sequences $\M$, that is, sequences $\M$ that are log-convex, of moderate growth and strongly nonquasianalytic,  whose growth index $\gamma(\M)$ is strictly greater than 1 (see Section \ref{sect-prelim} for the definition of these conditions and $\gamma(\M)$). Conversely, it was proven in \cite{L-S09} that if $\M$ is strongly regular,  $\mathcal{M}:\mathcal{S}_{\M}(0,\infty)\to\Lambda_{\M}$ is surjective and
\begin{eqnarray}\label{condgammaM}
\sum_{p=0}^\infty\Big(\frac{M_p}{M_{p+1}}\Big)^{1/\ga(\M)}=\infty,
\end{eqnarray}
then $\gamma(\M)>1$.

Our aim is to improve and complete these results by including the spaces $\mathcal{S}_{\M}^{\A}(0,\infty)$ in our considerations, by dropping some hypotheses on $\M$, specially moderate growth and~\eqref{condgammaM}, and by also studying the injectivity of the Stieltjes moment mapping. Our key tools are: a better understanding of the meaning of the different growth conditions usually imposed on the sequence $\M$ and their expression in terms of indices of O-regular variation, as developed in~\cite{JimenezSanzSchindlIndex}; the use of the Fourier transform in order to translate our problems into the corresponding ones for the asymptotic Borel mapping in certain ultraholomorphic classes on the upper half-plane; the enhanced information obtained in~\cite{JG-S-S} about the injectivity and surjectivity of the asymptotic Borel mapping for sequences $\M$ subject to minimal conditions.

The paper is organized as follows. In the preliminary Section \ref{sect-prelim} we gather the main facts needed regarding sequences, ultraholomorphic classes, the asymptotic Borel mapping $\mathcal{B}$ and growth indices related to the injectivity and surjectivity of $\mathcal{B}$, Gelfand-Shilov spaces and the Laplace transform.
Section~\ref{sect-Stiel-mom-problem} contains the main results. Firstly, in Theorem~\ref{inj-stiel}, we characterize the injectivity of the Stieltjes moment mapping $\mathcal{M}$ (defined on either $\mathcal{S}_{\M}^{\A}(0,\infty)$ or $\mathcal{S}_{\M}(0,\infty)$)  by an easy condition on the sequence $\M$, under minimal conditions on both $\M$ and $\A$. In Theorem~\ref{surj-stiel-modified} the surjectivity of $\mathcal{M}$ is characterized by the condition $\gamma(\M)>1$, for  $\M$ log-convex and of moderate growth and $\A$ weakly log-convex and non-quasianalytic. In particular, this result significantly improves those in~\cite{L-S09} and, moreover, extends to a general situation the statement of surjectivity of $\mathcal{M}$ in the case of the space $\mathcal{S}_{(p!^{\alpha-1})}^{(p!^{\beta-1})}(0,\infty)$, with $\alpha>2$ and $\beta>1$, that appeared (without proof) in~\cite[Thm.\ 3.3]{ChungKimYeom} and which is, up to the best of our knowledge, the only known result dealing with spaces of the type $\mathcal{S}_{\M}^{\A}(0,\infty)$. If moderate growth for $\M$ is substituted by the weaker condition of derivation closedness, we are only able to prove that $\gamma(\M)>1$ is necessary for the surjectivity of $\mathcal{M}$. We conclude this section by showing that the Stieltjes moment mapping is never bijective and that there exist strongly regular  sequences for which $\mathcal{M}$  is neither injective nor surjective. The final Section \ref{sect-mom_prob-orig} is devoted to the study of a related moment problem ``at the origin''. More precisely, we consider the space $\mathcal{D}^{\M}(0,1)$ consisting of all $\varphi \in C^\infty(\R)$ with $\supp \varphi \subseteq [0,1]$ such that there exists $h>0$ with
$$
\sup_{p \in \N} \sup_{x \in [0,1]} \frac{|\varphi^{(p)}(x)|}{h^{p}p!M_p} < \infty
$$
and we define
$$
\mu^0_p(\varphi) = \int_0^1 \frac{\varphi(x)}{x^p} \dx,  \qquad p \in \N.
$$
The study of the injectivity and surjectivity of the mapping
$\mathcal{M}^0: \DM(0,1) \rightarrow \Lambda_{\widecheck\M}: \varphi \rightarrow (\mu^0_p(\varphi))_p$, where $\widecheck\M=(M_p/p!)_{p\in\N}$, is reduced to the one of $\mathcal{M}$ via a suitable use of the Fourier transform.

\section{preliminaries}\label{sect-prelim}
\subsection{Weight sequences}\label{subsect-seq}

We set $\N=\{0,1,2,\dots\}$.
Throughout this article $\M = (M_p)_{p \in \N}$  will stand for a sequence of positive real numbers with $M_0 = 1$. We define $m_p = M_{p+1}/M_p$, $p \in \N$. The sequence $\M$ is said to be a \emph{weight sequence} if $m_p \rightarrow \infty$ as $p \to \infty$. Furthermore, we set $\widehat{\M} = (p!M_p)_{p \in \N}$ and $\widecheck{\M} = (M_p/p!)_{p \in \N}$. We shall use the following conditions on  sequences $\M$:

\begin{itemize}
\item[$\lc$] $\M$ is \emph{log-convex} if $M^2_p \leq M_{p-1}M_{p+1}$, $p \in \Z_+=\{1,2,\dots\}$.
\item[$\wlc$] $\M$ is \emph{weakly log-convex} if $\widehat{\M}$ satisfies $\lc$.
\item[$\dc$] $\M$ is \emph{derivation closed} if $M_{p+1} \leq C_0H^{p+1}M_p$, $p \in \N$, for some $C_0,H \geq 1$.
\item[$\mg$] $\M$ has \emph{moderate growth} if $M_{p+q} \leq C_0H^{p+q}M_pM_q$, $p,q \in \N$, for some $C_0,H \geq 1$.
\item[$\nq$] $\M$ is \emph{non-quasianalytic} if $\displaystyle \sum_{p=0}^\infty \frac{1}{(p+1)m_p} < \infty$.
\item[$\snq$] $\M$ is \emph{strongly non-quasianalytic} if $\displaystyle \sum_{q=p}^\infty \frac{1}{(q+1)m_q} \leq  \frac{C}{m_p}$, $p \in \N$, for some $C > 0$.
\end{itemize}

\begin{remark}\label{rem_propertiesM}
All these properties are preserved when passing from $\M$ to $\widehat{\M}$. In particular, a sequence satisfying $\lc$ is also $\wlc$. However, only $\dc$ and $\mg$ are generally kept when going from $\M$ to $\widecheck{\M}$.
\end{remark}

A sequence $\M$ satisfying $\wlc$ and $\nq$ is easily proved to be a weight sequence. A sequence $\M$ is said to be \emph{strongly regular} if it satisfies $\lc$, $\mg$ and $\snq$ (so, $\M$ is a weight sequence). The Gevrey sequence $(p!^{\alpha})_{p}$ ($\alpha>0$) is strongly regular.

In the classical work of H.~Komatsu~\cite{Komatsu}, the properties $\lc$, $\dc$ and $\mg$ are denoted by $(M.1)$, $(M.2)'$ and $(M.2)$, respectively, while $\nq$ and $\snq$ for $\M$ are the same as properties $(M.3)'$ and $(M.3)$ for $\widehat{\M}$, respectively.

For later use, we recall that $\lc$ (together with $M_0=1$) implies that \begin{equation}\label{eq-conseq-lc}
M_jM_p\le M_{j+p},\quad j,p\in\N.
\end{equation}
Following Komatsu~\cite{Komatsu}, the relation $\M \subset \NN$ between two  sequences means that
there are $C,h > 0$ such that $M_p \leq Ch^pN_p$ for all $p \in \N$. The \emph{associated function} of $\M$ is defined as
$$
\omega_{\M}(t) := \sup_{p \in \N} \log \frac{t^p}{M_p}, \qquad t > 0,
$$
and $\omega_{\M}(0) := 0$.

The following technical lemma shall be used later on.
\begin{lemma} \label{dc-below} Let $\M$ be a (weight) sequence satisfying $\wlc$ and $\nq$. Then, there is a (weight) sequence $\NN$ with $\NN \subset  \M$ satisfying $\wlc$, $\dc$ and $\nq$.
\end{lemma}

\begin{proof}
Define $a_p = \min\{ 2^{p+1}, (p+1)m_p\}$, $p \in \N$,
and $N_0 = 1$; $N_p = \frac{1}{p!} \prod_{j = 0}^{p-1} a_j$, $p \in \Z_+$. It is straightforward to check that ${\NN} = (N_p)_{p \in \N}$ satisfies all the requirements.
\end{proof}

\subsection{Ultraholomorphic classes on the upper half-plane and the asymptotic Borel mapping}\label{subsect-ultrahol}
We write $\HH$ for the open upper half-plane of the complex plane $\C$ and, given an open set $\Omega \subseteq \C$, we denote by
$\mathcal{O}(\Omega)$ the space of holomorphic functions in $\Omega$. Let $\M$ be a weight sequence. For $h>0$
we define $\mathcal{A}_{\M,h}(\HH)$ as the space consisting of all $f \in \mathcal{O}(\HH)$ such that
$$
\sup_{p \in \N} \sup_{z \in \HH} \frac{|f^{(p)}(z)|}{h^pp!M_p} < \infty.
$$
We set $\mathcal{A}_{\M}(\HH) = \bigcup_{h>0} \mathcal{A}_{\M,h}(\HH)$. Next, for $h>0$ we define 
$\mathcal{E}^{\M,h}_{\infty}(\R)$ 
as the space consisting of all $f \in C^\infty(\R)$ such that
$$
\sup_{p \in \N} \sup_{x \in \R} \frac{|f^{(p)}(x)|}{h^pp!M_p} < \infty.
$$
We set $\mathcal{E}^{\M}_{\infty}(\R)= \bigcup_{h>0}\mathcal{E}^{\M,h}_{\infty}(\R)$.

The following result is standard; it follows from the fact that the elements of $\mathcal{A}_{\M,h}(\HH)$ together with all their derivatives are Lipschitz on $\HH$.

\begin{lemma}\label{extension}
Let $\M$ be a weight sequence and let $f \in \mathcal{A}_{\M,h}(\HH)$ for some $h > 0$. Then,
\begin{equation*}
f_p(x) = \lim_{z \in \HH, z \to x} f^{(p)}(z)  \in \C
\end{equation*}
exists for all $x \in \R$ and $p \in \N$. Moreover, $f_0 \in C^\infty(\R)$ and $(f_0)^{(p)} = f_p$ for all $p \in \N$.
Consequently, $f_0 \in \mathcal{E}^{\M,h}_{\infty}(\R)$.
\end{lemma}

\begin{remark}\label{extension-rem}
 Let $\M$ be a weight sequence and let  $f \in \mathcal{A}_{\M}(\HH)$. In the sequel, we shall simply write
$$
f(x) = \lim_{z \in \HH, z \to x} f(z), \qquad    x \in \R.
$$
Lemma \ref{extension} states that $f$ is continuous on $\overline{\HH}$, $f_{|\R} \in \mathcal{E}^{\M}_{\infty}(\R)$ and that
$$
f^{(p)}(x) = \lim_{z \in \HH, z \to x} f^{(p)}(z)
$$
for all $x \in \R$ and $p \in \N$.
\end{remark}

Let $\M$ be a weight sequence. For $h>0$ we define $\Lambda_{\M,h}$  as the space consisting of all sequences $(c_p)_p \in \C^\N$ such that
$$
\sup_{p \in \N} \frac{|c_p|}{h^pp!M_p} < \infty.
$$
We set $\Lambda_{\M} = \bigcup_{h > 0} \Lambda_{\M,h}$. The \emph{asymptotic Borel mapping} is defined as
$$
\mathcal{B}: \mathcal{A}_{\M}(\HH) \rightarrow \Lambda_{\M}: f \rightarrow (f^{(p)}(0))_{p},
$$
which is well-defined by Lemma \ref{extension} (see also Remark \ref{extension-rem}).
For a fairly complete account on the injectivity and surjectivity of the asymptotic Borel mapping on various ultraholomorphic classes defined on arbitrary sectors we refer to \cite{JG-S-S}. There, two indices $\ga(\M)$ and $\om(\M)$, associated to the sequence $\M$, play a prominent role.
In~\cite[Ch.~2]{PhDJimenez} and \cite[Sect.~3]{JimenezSanzSchindlIndex}, the connections between these indices, the growth properties usually imposed on  sequences, and the theory of O-regular variation, have been thoroughly studied. We summarize here some facts. The first index, introduced by V. Thilliez~\cite[Sect.\ 1.3]{thilliez} for strongly regular sequences, may be defined for any weight sequence $\M$ satisfying $\lc$ as
$$\ga(\M):=\sup\{\mu>0\, | \,(m_{p}/(p+1)^\mu)_{p}\hbox{ is almost increasing} \}\in[0,\infty]
$$
(a sequence $(c_p)_{p}$ is \emph{almost increasing} if there exists $a>0$ such that $c_p\leq a c_q $ for all $ q\geq p$). On the other hand,
for  $\beta>0$ we say that $\M$  satisfies $(\ga_{\beta})$ if there is $C>0$ such that
\begin{equation*}
(\ga_{\beta})\qquad \sum^\infty_{q=p} \frac{1}{(m_q)^{1/\beta}}\leq \frac{C (p+1) }{(m_p)^{1/\beta}},  \qquad p\in\N.
\end{equation*}
Then one has that
$$ \ga(\M)=\sup\{\beta>0  \, | \, \M \,\, \text{satisfies} \,\, (\ga_{\beta})  \}.$$
Moreover, the following statements hold:
\begin{enumerate}
 \item[(i)] $\ga(\M)>0$ if and only if $\M$ satisfies $\snq$.
 \item[(ii)] $\ga(\M)>\beta$ if and only if $\M$ satisfies $(\ga_\beta)$.
\end{enumerate}

The surjectivity of the asymptotic Borel mapping can be characterized as follows.
\begin{theorem} \emph{(\cite[Thm.\ 4.17]{JG-S-S})}\label{surj-borel} Let $\M$ be a strongly regular weight sequence. Then, $\mathcal{B}: \mathcal{A}_{\M}(\HH) \rightarrow \Lambda_{\M}$ is surjective if and only if
$$
\sup_{p \in \N} \frac{m_p}{p+1}\sum_{q = p}^\infty \frac{1}{m_q} < \infty,
$$
or, equivalently, $\ga(\M)>1$.
\end{theorem}

The second index $\om(\M)$ is given by
$$\om(\M):= \displaystyle\liminf_{p\to\infty} \frac{\log(m_{p})}{\log(p)}\in[0,\infty],
$$
and it turns out that
\begin{align*}
\om(\M)&=\sup\{\mu>0\, | \, \sum^\infty_{p=0} \frac{1}{(m_p)^{1/\mu}}<\infty \}\\
&=\sup\{\mu>0 \, | \, \sum^\infty_{p=0} \frac{1}{((p+1)m_p)^{1/(\mu+1)}} < \infty \}.
\end{align*}

Concerning the injectivity of the asymptotic Borel mapping, we have the next result.
\begin{theorem} \emph{(\cite[Thm.\ 12]{Salinas}, \cite[Thm.\ 3.4]{JG-S-S})}\label{inj-borel} Let $\M$ be a weight sequence satisfying $\lc$. Then, $\mathcal{B}: \mathcal{A}_{\M}(\HH) \rightarrow \Lambda_{\M}$ is injective if and only if
$$
\sum_{p = 0}^\infty \frac{1}{((p+1)m_p)^{1/2}} = \infty,
$$
which in turn implies that $\om(\M)\le 1$.
\end{theorem}

Finally, we mention that if $\M$ is a weight sequence satisfying $\lc$, the asymptotic Borel mapping  $\mathcal{B}: \mathcal{A}_{\M}(\HH) \rightarrow \Lambda_{\M}$ is not bijective \cite[Thm.\ 3.17]{JG-S-S}.

\subsection{Gelfand-Shilov spaces}\label{GSspaces} Let $\M$ and $\A$ be weight sequences. For $h>0$ we
define $\mathcal{S}_{\M,h}^{\A,h}(\R)$ as the space consisting of all $\varphi \in C^\infty(\R)$ such that
$$
\sup_{p,q \in \N} \sup_{x \in \R} \frac{|x^p\varphi^{(q)}(x)|}{h^{p+q}p!M_pq!A_q} < \infty.
$$
Notice that $\varphi \in C^\infty(\R)$ belongs to  $\mathcal{S}_{\M,h}^{\A,h}(\R)$ if and only if
$$
\sup_{q \in \N} \sup_{x \in \R} \frac{|\varphi^{(q)}(x)|e^{\om_{\widehat{M}}(|x|/h)}}{h^{q}q!A_q} < \infty.
$$
We set $\SMA(\R) = \bigcup_{h>0} \mathcal{S}_{\M,h}^{\A,h}(\R)$. Analogously, we define  $\mathcal{S}_{\M,h}(\R)$, $h> 0$, as the space consisting of all $\varphi \in C^\infty(\R)$ such that, for all $q\in\N$,
$$
\sup_{p\in \N} \sup_{x \in \R} \frac{|x^p\varphi^{(q)}(x)|}{h^{p}p!M_p} < \infty
$$
and set $\mathcal{S}_{\M}(\R) = \bigcup_{h>0} \mathcal{S}_{\M,h}(\R)$. As in the introduction, we define
$$
\SMA(0,\infty) := \{ \varphi \in \SMA(\R) \, | \, \supp \varphi \subseteq [0,\infty)\}
$$
and
$$
\mathcal{S}_{\M}(0,\infty) := \{ \varphi \in \mathcal{S}_{\M}(\R) \, | \, \supp \varphi \subseteq [0,\infty)\}.
$$
Recall that $\SMA(0,\infty)\subset \mathcal{S}_{\M}(0,\infty)$. Suppose that $\A$ satisfies $\wlc$, then $\SMA(0,\infty)$ is non-trivial if and only if $\A$ satisfies $\nq$, as follows from the Denjoy-Carleman theorem.

In the remainder of this subsection we determine the image of the spaces $\SMA(\R)$ and $\SMA(0,\infty)$
under the Fourier transform (cf.\ \cite[Sect.\ IV.6]{G-S}). We fix the constants in the
Fourier transform as follows
$$
\mathcal{F}(\varphi)(\xi) = \widehat{\varphi}(\xi) = \int_{-\infty}^\infty \varphi(x) e^{ix\xi} \dx, \qquad \varphi \in L^1(\R).
$$
\begin{proposition}\label{Fourier-char} Let $\M$ and $\A$ be weight sequences satisfying $\wlc$ and $\dc$. Then, the Fourier transform is an isomorphism from $\SMA(\R)$ onto $\SAM(\R)$.
\end{proposition}
\begin{proof}
Since the Fourier transform is an isomorphism on the Schwartz space $\mathcal{S}(\R)$ and $\mathcal{F}^{-1}(\varphi)(\xi) = (2\pi)^{-1} \mathcal{F}(\varphi) (- \xi)$
 for all $\varphi \in \mathcal{S}(\R)$, it suffices to show that $\mathcal{F}(\SMA(\R))  \subseteq \SAM(\R)$. Let $h \geq 1$ and $\varphi \in \mathcal{S}_{\M,h}^{\A,h}(\R)$ be arbitrary. Choose $C > 0$ such that
$$
\sup_{x \in \R} |x^p\varphi^{(q)}(x)| \leq Ch^{p+q}p!M_pq!A_q, \qquad p,q \in \N.
$$
Since $\M$ and $\A$ are weight sequences they are both increasing from some term on, and so there exists $D\ge 1$ such that $M_j\le DM_p$ and $A_j\le DA_p$ for all $j\le p$.
Hence,
\begin{align*}
\sup_{x \in \R} (1+|x|)^p|\varphi^{(q)}(x)| &\leq \sum_{j=0}^p \binom{p}{j} \sup_{x \in \R} |x^j\varphi^{(q)}(x)| \\
&\leq C \sum_{j=0}^p \binom{p}{j}h^{j+q}j!M_jq!A_q\\
&\leq CD(2h)^{p+q}p!M_pq!A_q
\end{align*}
for all $p,q \in \N$. Therefore,
\begin{align*}
\sup_{\xi \in \R} |\xi^q\widehat{\varphi}^{(p)}(\xi)| &\leq \sum_{j=0}^{\min\{p,q\}} \binom{q}{j} \frac{p!}{(p-j)!} \int_{-\infty}^\infty |x^{p-j} \varphi^{(q-j)}(x)| \dx \\
&\leq \sum_{j=0}^{\min\{p,q\}} \binom{q}{j}\binom{p}{j} j! \int_{-\infty}^\infty \frac{(1+|x|)^{p+2-j} |\varphi^{(q-j)}(x)|}{(1+|x|)^2} \dx \\
&\leq 8CDh^2 \sum_{j=0}^{\min\{p,q\}} \binom{q}{j}\binom{p}{j} j! (2h)^{p+q -2j}(p+2-j)!M_{p+2-j}(q-j)!A_{q-j} \\
&\leq 8Ch^2D^3(4h)^{p+q}(p+2)!M_{p+2}q!A_{q} \\
&\leq 32C_0^2H^3Ch^2D^3(8H^2h)^{p+q}p!M_{p}q!A_{q} \\
\end{align*}
for all $p,q \in \N$.
\end{proof}

In view of Proposition \ref{Fourier-char}, the next result can be shown in a similar way as \cite[Prop.\ 2.1]{C-C-K}.
\begin{proposition}\label{Fourier-char-supp} Let $\M$ and $\A$ be weight sequences satisfying $\wlc$ and $\dc$. Let $\psi: \R \rightarrow \C$. Then, $\psi \in \mathcal{F}(\SMA(0,\infty))$ if and only if $\psi \in \SAM(\R)$ and there is $\Psi: \overline{\HH} \rightarrow \C$ satisfying the following conditions:
\begin{itemize}
\item[$(i)$] $\Psi_{|\R} = \psi$.
\item[$(ii)$] $\Psi$ is continuous on $\overline{\HH}$ and analytic on $\HH$.
\item[$(iii)$] $\lim_{\zeta \in \overline{\HH}, \zeta \to \infty} \Psi(\zeta) = 0$.
\end{itemize}
\end{proposition}

\subsection{The Laplace transform}\label{sect-Laplace} Let $\M$ be a weight sequence. We define $C_{\M,h}[0,\infty)$ as the space consisting of all $\varphi \in C([0,\infty))$ such that
$$
\sup_{p \in \N} \sup_{x \in [0,\infty)} \frac{x^p|\varphi(x)|}{h^{p}p!M_p} < \infty
$$
or, in other words, such that
$$
\sup_{x \in [0,\infty)} |\varphi(x)|e^{\om_{\widehat{M}}(|x|/h)} < \infty.
$$
We set $\CM = \bigcup_{h>0} C_{\M,h}[0,\infty)$. The \emph{Laplace transform} of $\varphi \in \CM$ is defined as
$$
\mathcal{L}(\varphi)(\zeta) =  \int_0^\infty \varphi(x)e^{ix\zeta}\dx, \qquad \zeta \in \overline{\HH}.
$$
\begin{remark}
Let $\M$ and $\A$ be weight sequences. We may view $\SMA(0,\infty)$ and $\mathcal{S}_{\M}(0,\infty)$ as subspaces of $\CM$. Notice that $\mathcal{L}(\varphi)_{|\R} = \widehat{\varphi}$ for all $\varphi \in \SMA(0,\infty)$.
\end{remark}
\begin{lemma}\label{inj-lapl}
Let $\M$ be a weight sequence satisfying $\dc$. Then, the mapping
$
\mathcal{L} : \CM \rightarrow \mathcal{A}_{\M}(\HH)
$
is well-defined and injective.
\end{lemma}
\begin{proof}
The fact that $\mathcal{L}$ is well-defined follows along the same lines as the proof of Proposition \ref{Fourier-char}. 
We now show that  $\mathcal{L}$ is injective. Let $\varphi \in \CM$ be
such that $\mathcal{L}(\varphi) \equiv 0$ on $\HH$. Since $\mathcal{L}(\varphi)$ is continuous on $\overline{\HH}$, we also have that $\mathcal{L}(\varphi) \equiv 0$ on $\R$.  Define
$$
\widetilde{\varphi}(x) =
\left\{
	\begin{array}{ll}
		\mbox{$\varphi(x)$},  &  x \geq 0, \\ \\
		\mbox{$0$},  &  x < 0.
	\end{array}
\right.
$$
Then, $\widetilde{\varphi} \in L^1(\R)$  and $\mathcal{F}(\widetilde{\varphi}) =  \mathcal{L}(\varphi)_{|\R} \equiv 0$. Since $\mathcal{F}$ is injective on $L^1(\R)$,  $\widetilde{\varphi} = 0$  almost everywhere. As $\varphi$ is continuous on $[0,\infty)$, we may conclude that $\varphi \equiv 0$ on $[0,\infty)$.
\end{proof}

\section{The Stieltjes moment problem in Gelfand-Shilov spaces}\label{sect-Stiel-mom-problem}
Let $\M$ be a weight sequence. The $p$-th moment, $p \in \N$, of an element $\varphi \in \CM$ is defined as
$$
\mu_p(\varphi) := \int_0^\infty x^p \varphi(x) \dx.
$$
If $\M$ satisfies $\dc$, then the \emph{Stieltjes moment mapping}
$$
\mathcal{M}: \CM \rightarrow \Lambda_{\M}: \varphi \mapsto (\mu_p(\varphi))_p
$$
is well-defined. The goal of this section is to characterize
the injectivity and surjectivity of the Stieltjes moment mapping on $\CM$ and its
subspaces of type $\SMA(0,\infty)$ and $\mathcal{S}_{\M}(0,\infty)$ in terms of the defining weight sequence $\M$. We employ the
same idea as in \cite{D-E}, which was later also used in \cite{C-C-K, L-S08,L-S09}. Namely, we shall reduce these
problems to their counterparts for the asymptotic Borel mapping (Theorems \ref{inj-borel} and
\ref{surj-borel}) via the Laplace transform. In this regard, the following formula is fundamental
$$
\mathcal{L}(\varphi)^{(p)}(0) = i^p \mu_p(\varphi), \qquad \varphi \in \CM, p \in \N.
$$
 In the next lemma we construct an auxiliary function that shall
be frequently used throughout this section (compare with the function $G$ from \cite{D-E}). We
set $\HH_{-1} = \{ z \in \C \, | \, \Im m \, z > -1 \}$.

\begin{lemma}\label{aux} Let $\A$ be a weight sequence satisfying $\wlc$ and $\nq$. Then, there is $G \in \mathcal{O}(\HH_{-1})$ satisfying the following conditions:
\begin{itemize}
\item[$(i)$] $G$ does not vanish on $\HH_{-1}$.
\item[$(ii)$] $\displaystyle \sup_{z \in \HH_{-1}} |G(z)|e^{\om_{\widehat{A}}(|z|)} < \infty$.
\item[$(iii)$] $\displaystyle \sup_{p \in \N}\sup_{x \in \R} \frac{|G^{(p)}(x)|e^{\om_{\widehat{A}}(|x|/2)}}{2^pp!} < \infty$.
\end{itemize}
\end{lemma}

The construction of the function $G$ from Lemma \ref{aux} is based on the following result.

\begin{lemma} \emph{(\cite[Lemma 2.2]{B-M-T})} \label{Poisson} Let $\omega: [0,\infty) \rightarrow [0,\infty)$ be an increasing continuous function such that
$$
\int_{0}^\infty \frac{\omega(t)}{1+t^2} \dt < \infty
$$
and extend $\omega$ as an even function to the whole real line.
Then, the Poisson transform of $\omega$ on $\HH$ given by
$$
P_\omega(z) = \frac{y}{\pi} \int_{-\infty}^\infty \frac{\omega(t)}{(t-x)^2 + y^2} \dt,  \qquad z = x+iy \in \HH,
$$
is harmonic and positive on $\HH$ and satisfies
$$
P_\omega(z) \geq \frac{1}{4} \omega(|z|), \qquad z \in \HH.
$$
\end{lemma}
\begin{proof}[Proof of Lemma \ref{aux}] Set $\omega = \om_{\widehat{A}}(2\,\cdot)$. Since $A$ satisfies $\wlc$ and $\nq$, we have that \cite[Lemma 4.1]{Komatsu}
$$
\int_{0}^\infty \frac{\omega(t)}{1+t^2} \dt < \infty.
$$
Write $U = 4P_\omega$ (cf.\ Lemma \ref{Poisson}) and let $V$ be the harmonic conjugate of $U$ on $\HH$. Define $G = e^{-(U( \, \cdot \, + i) + iV( \, \cdot \, + i))}$. It is clear that $G \in \mathcal{O}(\HH_{-1})$ and that $(i)$ is satisfied. We now show $(ii)$ and $(iii)$.

$(ii)$: For $z \in \HH_{-1}$ with $|z| \geq 2$ we have that $2|z+i| \geq |z|$ and, thus,
$$
|G(z)| = e^{-U(z+i)} \leq e^{-\om_{\widehat{A}}(2|z+i|)} \leq  e^{-\om_{\widehat{A}}(|z|)}.
$$
For $z \in \HH_{-1}$ with $|z| \leq 2$ we have that
$$
|G(z)| \leq e^{\om_{\widehat{A}}(2)} e^{-\om_{\widehat{A}}(|z|)}.
$$

$(iii)$: By the Cauchy estimates and $(ii)$ there is $C > 0$ such that
$$
|G^{(p)}(x)| \leq 2^p p! \max_{|z-x| \leq 1/2} |G(z)| \leq C 2^p p! \max_{|z-x| \leq 1/2} e^{-\om_{\widehat{A}}(|z|)}
$$
for all $x \in \R$ and $p \in \N$. For $x \in \R$ with $|x| \geq 1$ we have that $|z| \geq |x|/2$ for all $z \in \C$ with $|z-x| \leq 1/2$. Hence,
$$
|G^{(p)}(x)| \leq C 2^p p! e^{-\om_{\widehat{A}}(|x|/2)}, \qquad p \in \N.
$$
For $x \in \R$ with $|x| \leq 1$ we have that
$$
|G^{(p)}(x)| \leq Ce^{\om_{\widehat{A}}(1/2)} 2^p p! e^{-\om_{\widehat{A}}(|x|/2)}, \qquad p \in \N.
$$
\end{proof}
Proposition \ref{Fourier-char-supp} and Lemma \ref{aux} imply the following important lemma.
\begin{lemma} \label{multiplication} Let $\M$ be a weight sequence satisfying $\wlc$ and $\dc$
and let $\A$ be a weight sequence satisfying $\wlc$, $\dc$ and $\nq$.
Consider the function $G$ from Lemma \ref{aux}. Then, $fG \in \mathcal{F}(\SMA(0,\infty))$ for all $f \in \mathcal{A}_{\M}(\HH)$.
\end{lemma}

We are ready to study the injectivity and surjectivity of the Stieltjes moment mapping.

\begin{theorem}\label{inj-stiel} Let $\M$ be a weight sequence satisfying $\lc$ and $\dc$ and let $\A$ be a
weight sequence satisfying $\wlc$ and $\nq$. Then, the following statements are equivalent:
\begin{itemize}
\item[$(i)$] $\displaystyle \sum_{p = 0}^\infty \frac{1}{((p+1)m_p)^{1/2}} = \infty$.
\item[$(ii)$] $\mathcal{B}: \mathcal{A}_{\M}(\HH) \rightarrow \Lambda_{\M}$ is injective.
\item[$(iii)$] $\mathcal{M}: \CM \rightarrow \Lambda_{\M}$ is injective.
\item[$(iv)$] $\mathcal{M}: \mathcal{S}_{\M}(0,\infty) \rightarrow \Lambda_{\M}$ is injective.
\item[$(v)$] $\mathcal{M}: \SMA(0,\infty) \rightarrow \Lambda_{\M}$ is injective.
\end{itemize}
\end{theorem}
\begin{proof}
$(i) \Rightarrow (ii)$: By Theorem \ref{inj-borel}.

$(ii) \Rightarrow (iii)$: Let $\varphi \in \CM$ be such that $\mu_p(\varphi) = 0$ for all $p \in \N$. By Lemma \ref{inj-lapl} we have that $\mathcal{L}(\varphi) \in \mathcal{A}_{\M}(\HH)$.  Moreover,
$
\mathcal{L}(\varphi)^{(p)}(0) = i^p \mu_p(\varphi) = 0$ for all $p \in \N$ and, thus,
$\mathcal{L}(\varphi) \equiv 0$. Since $\mathcal{L}$ is injective (Lemma \ref{inj-lapl}), we obtain that $\varphi \equiv 0$.

$(iii) \Rightarrow (iv) \Rightarrow (v)$: Obvious.

$(v) \Rightarrow (i)$: By Lemma \ref{dc-below} we may assume that $\A$ satisfies $\dc$. In view of Theorem \ref{inj-borel} it suffices to show that $\mathcal{B}: \mathcal{A}_{\M}(\HH) \rightarrow \Lambda_{\M}$ is injective. Let $f \in \mathcal{A}_{\M}(\HH)$ be such that $f^{(p)}(0) = 0$ for all $p \in \N$. Consider the function $G$ from Lemma \ref{aux}. By Lemma \ref{multiplication} we have that $fG = \widehat{\varphi}$ for some $\varphi \in \SMA(0,\infty)$. Observe that
$$
\mu_p(\varphi) = (-i)^p \widehat{\varphi}^{(p)}(0) = (-i)^p(fG)^{(p)}(0) = (-i)^p\sum_{j=0}^p \binom{p}{j} f^{(j)}(0) G^{(p-j)}(0) = 0, \quad p \in \N.
$$
Hence, $\varphi \equiv 0$ and, thus, $fG \equiv 0$. Since $G$ does not vanish (Lemma \ref{aux}$(i)$), we obtain that $f \equiv 0$.
\end{proof}

\begin{theorem} \label{surj-stiel-modified} Let $\M$ be a weight sequence satisfying $\lc$ and $\dc$ and let $\A$ be a weight sequence satisfying $\wlc$ and $\nq$. Then, the following statements are equivalent:
\begin{itemize}
\item[$(i)$] $\mathcal{M}: \SMA(0,\infty) \rightarrow \Lambda_{\M}$ is surjective.
\item[$(ii)$] $\mathcal{M}: \mathcal{S}_{\M}(0,\infty) \rightarrow \Lambda_{\M}$ is surjective.
\item[$(iii)$] $\mathcal{M}: \CM \rightarrow \Lambda_{\M}$ is surjective.
\item[$(iv)$] $\mathcal{B}: \mathcal{A}_{\M}(\HH) \rightarrow \Lambda_{\M}$ is surjective.
\end{itemize}
Each of the previous statements implies the next one:
\begin{itemize}
\item[$(v)$] $\displaystyle \sup_{p \in \N} \frac{m_p}{p+1}\sum_{q = p}^\infty \frac{1}{m_q} < \infty$ or, equivalently, $\gamma(\M)>1$.
\end{itemize}
If, in addition, $\M$ satisfies $\mg$, then all the previous statements are equivalent.
\end{theorem}

In the proof of Theorem \ref{surj-stiel-modified} we shall use the following lemma (cf.\ \cite{D-E}).
\begin{lemma}\label{inversion}
Let $(c_p)_p \in \C^\N$ and let $G \in C^\infty((-\delta,\delta))$, for some $\delta > 0$, such that $G(0) \neq 0$. Set
$$
b_p = \sum_{j = 0}^p \binom{p}{j} c_j \left ( \frac{1}{G} \right)^{(p-j)} (0), \qquad p \in \N.
$$
Then,
$$
\sum_{j=0}^p \binom{p}{j} b_j G^{(p-j)}(0) = c_p, \qquad p\in \N.
$$
\end{lemma}
\begin{proof}
Choose $0 < \delta_1 \leq \delta$ such that $G$ does not vanish on $(-\delta_1,\delta_1)$. By E.\ Borel's theorem there is $f \in C^\infty((-\delta_1, \delta_1))$ such that $f^{(p)}(0) = c_p$ for all $p \in \N$. Set $g = f/G \in C^\infty((-\delta_1,\delta_1))$. Then,
$$
g^{(p)}(0) = \sum_{j = 0}^p \binom{p}{j} f^{(j)}(0) \left ( \frac{1}{G} \right)^{(p-j)} (0) = b_p, \qquad p \in \N.
$$
Hence,
$$
c_p = f^{(p)}(0) = (gG)^{(p)}(0) =  \sum_{j = 0}^p \binom{p}{j} g^{(j)}(0) G^{(p-j)} (0) =  \sum_{j = 0}^p \binom{p}{j} b_j G^{(p-j)} (0), \qquad p \in \N.
$$
\end{proof}

\begin{proof}[Proof of Theorem \ref{surj-stiel-modified}]
We first prove the equivalence of the statements $(i)$ to $(iv)$.

$(i) \Rightarrow (ii) \Rightarrow (iii)$: Obvious.

$(iii) \Rightarrow (iv)$: Let $(c_p)_{p} \in \Lambda_{\M}$ be arbitrary. Pick $\varphi \in \CM$ such that $\mu_p(\varphi) = (-i)^pc_p$ for all $p \in \N$. Then, $f = \mathcal{L}(\varphi)  \in \mathcal{A}_{\M}(\HH)$ (Lemma \ref{inj-lapl}) and $f^{(p)}(0) = i^p \mu_p(\varphi) = c_p$ for all $p \in \N$.

$(iv) \Rightarrow (i)$: By Lemma \ref{dc-below} we may assume that $\A$ satisfies $\dc$.  Let $(c_p)_{p} \in \Lambda_{\M}$ be arbitrary.  Consider the function $G$ from Lemma \ref{aux}. Set
$$
b_p = \sum_{j = 0}^p \binom{p}{j} i^jc_j \left ( \frac{1}{G} \right)^{(p-j)} (0), \qquad p \in \N.
$$
We claim that $(b_p)_{p} \in \Lambda_{\M}$ (cf.\ \cite[Prop.\ 6.4]{L-S09}). Indeed, choose $C,h > 0$ such that $|c_p| \leq C h^p p! M_p$ for all $p \in \N$. Next, since $1/G$ is holomorphic on a neighbourhood of the disk with center the origin and radius 1/2, there is $C' > 0$ such that $ |(1/G)^{(p)}(0)| \leq C' 2^p p!$ for all $p \in \N$. Hence,
$$
|b_p| \le C C'\sum_{j = 0}^p \binom{p}{j} h^j j!M_j 2^{p-j} (p-j)! \leq  CC' D(h +2)^p p!M_p, \qquad p \in \N,
$$
where $D\ge 1$ is chosen so that $M_j\le DM_p$ for all $j\le p$.
By assumption there is $f \in  \mathcal{A}_{\M}(\HH)$ such that $f^{(p)}(0) = b_p$ for all $p \in \N$. We have that $fG = \widehat{\varphi}$ for some $\varphi \in \SMA(0,\infty)$ by  Lemma \ref{multiplication}. Finally, Lemma \ref{inversion} implies that
$$
\mu_p(\varphi) = (-i)^p \widehat{\varphi}^{(p)}(0) = (-i)^p(fG)^{(p)}(0) = (-i)^p\sum_{j=0}^p \binom{p}{j} b_j G^{(p-j)}(0) = c_p, \qquad p \in \N.
$$

We  now prove the statements related to $(v)$. The implication $(iv) \Rightarrow (v)$  follows directly from \cite[Thm.\ 4.14$(i)$]{JG-S-S}.
If, in addition,  $\M$ satisfies $\mg$, condition $(v)$ implies that $\M$ satisfies $\snq$ as well (see Subsection~\ref{subsect-ultrahol}), and so $\M$ is strongly regular. Then, Theorem \ref{surj-borel} guarantees that $(iv)$ holds.
\end{proof}

\begin{corollary} \label{cor-bij} Let $\M$ be a weight sequence satisfying $\lc$ and $\dc$ and let $\A$ be a weight sequence satisfying $\wlc$ and $\nq$. Then, $\mathcal{M}: \CM \rightarrow \Lambda_{\M}$, $\mathcal{M}: \mathcal{S}_{\M}(0,\infty) \rightarrow \Lambda_{\M}$ and $\mathcal{M}: \SMA(0,\infty) \rightarrow \Lambda_{\M}$ are never bijective.
\end{corollary}
\begin{proof}
If any of the moment mappings were injective, we would have that $\sum_p ((p+1)m_p)^{-1/2}=\infty$ by Theorem~\ref{inj-stiel}. From Subsection~\ref{subsect-ultrahol} we deduce that $\omega(\M)\le 1$, which in turn implies that $\gamma(\M)\le 1$ because $\ga(\NN)\leq\om(\NN)$ for any weight sequence $\NN$ satisfying $\lc$. Hence,  $(v)$ from Theorem~\ref{surj-stiel-modified} is violated and therefore none of $(i)-(iii)$ from  Theorem~\ref{surj-stiel-modified} can be satisfied.
\end{proof}

\begin{example}
There exist  strongly regular sequences for which the Stieltjes moment mapping is neither injective nor surjective. E.g., in \cite[Example 4.18, Remark 4.19]{JimenezSanzSchindlLCSNPO} (see also \cite[Example~2.2.26]{PhDJimenez})  the sequence $\M$ is defined via its sequence of quotients, $M_p=\prod_{j=0}^{p-1}m_j$, where
 $$
 m_0=1; \qquad m_p=e^{\delta_p/p}m_{p-1}=\exp\left(\sum^p_{k=1}\frac{\delta_k}{k} \right), \qquad p\in\Z_+,
 $$
and the sequence $(\delta_k)_{k \in \Z_+}$ still has to be determined. Consider the sequences
 $$k_{j}:=2^{3^j} <q_{j}:=k^2_j=2^{ 3^j 2}< k_{j+1}=2^{3^{j+1}}, \quad{j\in\N},$$
and choose $(\delta_k)_{k}$ as follows:
\begin{align*}
\delta_1&=\delta_2=2, \\
\delta_k&=3,\quad \text{if} \quad k\in \{k_j+1, \dots, q_j\}, j\in\N,\\
\delta_k&=2,\quad \text{if} \quad k\in \{q_j+1, \dots, k_{j+1}\}, j\in\N.
\end{align*}
One can prove that $\M$ is strongly regular and that $\ga(\M)=2<\om(\M)=5/2$. Then, the sequence $\M^{1/2}:=(M^{1/2}_p)_{p\in\N}$ is again strongly regular and $\ga(\M^{1/2})=1<5/4=\om(\M^{1/2})$. Hence, both the injectivity and surjectivity of the Stieltjes moment mapping are discarded.

Subsequently, in \cite{PhDJimenez} (see also \cite{JimenezSanzSchindlIndex}), a general procedure has been designed to obtain strongly regular sequences with preassigned positive values of $\ga(\M)$ and $\om(\M)$. In particular, one can choose strongly regular sequences $\M$ with $\ga(\M)\le 1<\om(\M)$ and thereby exclude both injectivity and surjectivity.
\end{example}

\section{A moment problem at the origin}\label{sect-mom_prob-orig}
Let $\M$ be a weight sequence. For $h>0$ we
define $\mathcal{D}^{\M,h}(0,1)$ as the space consisting of all $\varphi \in C^\infty(\R)$ with $\supp \varphi \subseteq [0,1]$ such that
$$
\| \varphi \|_{\M,h} := \sup_{p \in \N} \sup_{x \in [0,1]} \frac{|\varphi^{(p)}(x)|}{h^{p}p!M_p} < \infty.
$$
We set $\DM(0,1) = \bigcup_{h>0} \mathcal{D}^{\M,h}(0,1)$. Suppose that $\M$ satisfies $\wlc$, then $\DM(0,1)$ is non-trivial if and only if $\M$ satisfies $\nq$, as follows from the Denjoy-Carleman theorem. Notice that $\DM(0,1) \subset \SAM(0,\infty)$ for all weight sequences $\A$.
\begin{lemma}\label{taylor}
Let $\M$ be a weight sequence and let $\varphi \in \mathcal{D}^{\M,h}(0,1)$ for some $h > 0$. Then,
$$
|\varphi(x)| \leq \| \varphi\|_{\M,h} h^pM_px^p
$$
for all $x \in [0,1]$ and $p \in \N$.
\end{lemma}
\begin{proof}
Since $\varphi^{(j)}(0) = 0$ for all $j \in \N$, Taylor's theorem implies that
$$
|\varphi(x)| \leq \frac{1}{p!} \sup_{t \in [0,x]} |\varphi^{(p)}(t)| x^p \leq  \|\varphi\|_{\M,h} h^pM_px^p
$$
for all $x \in [0,1]$ and $p \in \N$.
\end{proof}

Let $\M$ be a weight sequence. For $\varphi \in \DM(0,1)$ we define
$$
\mu^0_p(\varphi) = \int_0^1 \frac{\varphi(x)}{x^p} \dx,  \qquad p \in \N.
$$
The mapping
$$
\mathcal{M}^0: \DM(0,1) \rightarrow \Lambda_{\widecheck\M}: \varphi \rightarrow (\mu^0_p(\varphi))_p
$$
is well-defined by Lemma \ref{taylor}. The goal of this section is to characterize
the injectivity and surjectivity of the mapping  $\mathcal{M}^0$ in terms of the defining weight sequence $\M$. We shall reduce these
problems to their counterparts for the Stieltjes moment mapping (Theorems \ref{inj-stiel} and
\ref{surj-stiel-modified}) via the following lemma.
\begin{lemma}\label{reduction}
Let $\M$ be a weight sequence and let $\varphi \in \mathcal{D}^{\M}(0,1)$. Then,
$$
\mu_p(\widehat{\varphi}) = i^{p+1} p! \mu^0_{p+1}(\varphi), \qquad p \in \N.
$$
\end{lemma}
\begin{proof}
For all $p \in \N$ we have that
\begin{align*}
&\mu_p(\widehat{\varphi}) = \int_0^\infty \xi^p \widehat{\varphi}(\xi) \dxi = i^p  \int_0^\infty \mathcal{F}(\varphi^{(p)})(\xi) \dxi
= - i^{p+1}  \int_0^\infty (\mathcal{F}(\varphi^{(p)}(x)/x))'(\xi) \dxi  \\
& =  i^{p+1} \mathcal{F}(\varphi^{(p)}(x)/x)(0)
 =  i^{p+1} \int_0^1 \frac{\varphi^{(p)}(x)}{x} \dx
=  i^{p+1} p! \int_0^1 \frac{\varphi(x)}{x^{p+1}} \dx
= i^{p+1} p! \mu^0_{p+1}(\varphi).
\end{align*}
\end{proof}
We are ready to characterize the injectivity and surjectivity of the mapping $\mathcal{M}^0$.
\begin{theorem}\label{inj-origin} Let $\M$ be a weight sequence satisfying $\lc$, $\dc$ and $\nq$. Then, the following statements are equivalent:
\begin{itemize}
\item[$(i)$] $\displaystyle \sum_{p = 0}^\infty \frac{1}{((p+1)m_p)^{1/2}} = \infty$.
\item[$(ii)$] $\mathcal{M}: \mathcal{F}(\DM(0,1)) \rightarrow \Lambda_{\M}$ is injective.
\item[$(iii)$] $\mathcal{M}^0: \DM(0,1) \rightarrow \Lambda_{\widecheck\M}$ is injective.
\end{itemize}
\end{theorem}
\begin{remark}
If we assume that $\M$ satisfies $\lc$ and it does not satisfy $\nq$, then $\DM(0,1)$ is trivial and $\sum_{p = 0}^\infty ((p+1)m_p)^{-1}= \infty$, and so the three previous statements clearly hold true.
This justifies the hypothesis $\nq$ in Theorem~\ref{inj-origin}, while condition $\dc$ is needed in order to apply our previous results about the moment mapping $\mathcal{M}$.
\end{remark}
The proof of Theorem \ref{inj-origin} is based on the next result.

\begin{proposition}\label{uncertainty}
Let $\varphi \in L^1(\R)$ with $\supp \varphi \subseteq [0, \infty)$. If $\widehat{\varphi}$ vanishes on a subset of $\R$ with positive Lebesgue measure, then $\varphi = 0$ almost everywhere.
\end{proposition}
Proposition \ref{uncertainty} follows directly from the Lusin-Privalov theorem, which we include here for the reader's convenience.

\begin{theorem}[of Lusin and Privalov~\cite{LusinPrivalov}]\label{thPrivalov}
Let $F\in\mathcal{O}(\HH)$ and suppose that there exists a set $A\subset \R$ with positive Lebesgue measure such that, for every $x\in A$ and $0<\delta<1$, it holds that
$$
\lim_{\genfrac{}{}{0pt}{1}{z\to x}{z\in S_{x,\delta}}}F(z)=0,
$$
where $S_{x,\delta}\subset \HH$ is the sector with vertex at $x$, vertical bisecting direction and opening $\pi\delta$. Then, $F \equiv 0$ on $\HH$.
\end{theorem}

\begin{proof}[Proof of Proposition \ref{inj-origin}]
$(i) \Rightarrow (ii)$: Follows directly from Theorem \ref{inj-stiel}, since it is clear that $\mathcal{F}(\DM(0,1))\subset \CM$.

$(ii) \Rightarrow (i)$: By Theorem \ref{inj-stiel} it suffices to show that $\mathcal{M}: \CM \rightarrow \Lambda_{\M}$ is injective. Let $\varphi \in \CM$ be such that $\mu_p(\varphi) = 0$ for all $p \in \N$. Set
$$
\widetilde{\varphi}(x) =
\left\{
	\begin{array}{ll}
		\mbox{$\varphi(x)$},  &  x \geq 0, \\ \\
		\mbox{$0$},  &  x < 0.
	\end{array}
\right.
$$
Next, by condition $\nq$, we can choose $\psi \in \DM(0,1)$ such that $\psi(x) \neq 0$ for all $x \in (0,1)$ . Then,
$$
\widetilde{\varphi} \ast \widehat{\psi}
= \mathcal{F}(\mathcal{L}(\varphi)_{|\R}(- \, \cdot) \cdot \psi)
\in \mathcal{F}(\DM(0,1)),
$$
as follows from Lemmas \ref{extension} and \ref{inj-lapl}. Moreover, we have that
$$
\mu_p(\widetilde{\varphi} \ast \widehat{\psi}) = \sum_{j = 0}^p \binom{p}{j} \mu_j(\varphi) \mu_{p-j}(\widehat{\psi}) = 0, \qquad p \in \N.
$$
Hence, $\widetilde{\varphi} \ast \widehat{\psi} \equiv 0$ and, thus, also $\mathcal{F}^{-1}(\widetilde{\varphi} \ast \widehat{\psi}) = 
\mathcal{L}(\widetilde{\varphi})_{|\R}(- \, \cdot) \cdot \psi\equiv 0 $. Since $\psi(x) \neq 0$ for all $x \in (0,1)$, we obtain that $\mathcal{F}(\widetilde{\varphi})(\xi) =  \mathcal{L}(\varphi)_{|\R}( \xi) = 0$ for all $\xi \in (-1,0)$. Proposition \ref{uncertainty} yields that $\widetilde{\varphi} = 0$ almost everywhere. As $\varphi$ is continuous on $[0,\infty)$, we may conclude that $\varphi \equiv 0$ on $[0,\infty)$.

$(ii) \Rightarrow (iii)$: Follows directly from Lemma \ref{reduction}.

$(iii) \Rightarrow (ii)$: Let $\varphi \in \DM(0,1)$ be such that $\mu_p(\widehat{\varphi}) = 0$ for all $p \in \N$. By Lemma \ref{reduction} we have that $\mu^0_p(\varphi) = 0$ for all $p \in \Z_+$. Condition $\dc$ implies that $\varphi' \in \DM(0,1)$. Moreover, we have that $\mu^0_0(\varphi') = \int_0^1 \varphi'(x)\dx =0$ and
\begin{equation}\label{moments_derivative}
\mu^0_p(\varphi') = \int_0^1 \frac{\varphi'(x)}{x^p} \dx = p \int_0^1 \frac{\varphi(x)}{x^{p+1}} \dx = p\mu^0_{p+1}(\varphi) = 0,\quad p\in\Z_{+}.
\end{equation}
Hence, $\varphi' \equiv 0$ and, since $\varphi$ is compactly supported, we obtain that $\varphi \equiv 0$.
\end{proof}

\begin{theorem} \label{surj-origin-modified} Let $\M$ be a weight sequence satisfying $\lc$ and $\dc$. Then, the following statements are equivalent:
\begin{itemize}
\item[$(i)$] $\mathcal{M}^0: \DM(0,1) \rightarrow \Lambda_{\widecheck\M}$ is surjective.
\item[$(ii)$] $\mathcal{M}: \mathcal{F}(\DM(0,1)) \rightarrow \Lambda_{\M}$ is surjective.
\end{itemize}
Each of the previous statements implies the next one:
\begin{itemize}
\item[$(iii)$] $\displaystyle \sup_{p \in \N} \frac{m_p}{p+1}\sum_{q = p}^\infty \frac{1}{m_q} < \infty$ or, equivalently, $\gamma(\M)>1$.
\end{itemize}
If, in addition, $\M$ satisfies $\mg$, then all the previous statements are equivalent.
\end{theorem}

We need a lemma in preparation of the proof of  Theorem \ref{surj-origin-modified}.
\begin{lemma}\label{inversion-2}
Let $\M$ be a weight sequence and let $\chi \in \CM$ be such that $\mu_0(\chi) \neq 0$. Define $G = \mathcal{L}(\chi)_{|\R}$ and notice that $G(0) = \mu_0(\chi) \neq 0$. For $(c_p)_p \in \C^\N$ we set
$$
b_p = (-i)^p \sum_{j = 0}^p \binom{p}{j} i^jc_j \left ( \frac{1}{G} \right)^{(p-j)} (0), \qquad p \in \N.
$$
Then,
$$
\sum_{j=0}^p \binom{p}{j} b_j \mu_{p-j}(\chi) = c_p, \qquad p\in \N.
$$
\end{lemma}
\begin{proof}
 Lemma \ref{inversion} implies that
$$
\sum_{j=0}^p \binom{p}{j} b_j \mu_{p-j}(\chi) = \sum_{j=0}^p \binom{p}{j} b_j (-i)^{p-j} G^{(p-j)}(0) = c_p, \qquad p \in \N.
$$
\end{proof}

\begin{proof}[Proof of Theorem \ref{surj-origin-modified}]
We first prove the equivalence of $(i)$ and $(ii)$.

$(i) \Rightarrow (ii)$: Follows directly from Lemma \ref{reduction}.

$(ii) \Rightarrow (i)$: Let $(c_p)_p \in \Lambda_{\widecheck{\M}}$ be arbitrary. Set $b_p = i^{p+1} (p-1)!c_p$, $p \in \N$. Then, $(b_p)_p \in \Lambda_{\M}$. Choose $\varphi \in \DM(0,1)$ such that $\mu_p(\widehat{\varphi}) = b_p$ for all $p \in \N$. Consider the function $\varphi'$, which belongs to $\DM(0,1)$ because of $\dc$. Then, the computation in~\eqref{moments_derivative} and Lemma \ref{reduction} imply that
$$
\mu_p^0(\varphi') = p\mu_{p+1}^0(\varphi) = \frac{p\mu_p(\widehat{\varphi})}{i^{p+1}p!} =
\frac{pb_p}{i^{p+1}p!} = c_p, \qquad p \in \N.
$$

For the second part of the theorem, observe that, since $\mathcal{F}(\DM(0,1))\subset \CM$, the implication $(ii) \Rightarrow (iii)$ follows directly from $(iii) \Rightarrow (v)$ in Theorem \ref{surj-stiel-modified}.

Finally, if $\M$ additionally  satisfies $\mg$ and we depart from $(iii)$, as before we first deduce that $\M$ is strongly regular and, thus, satisfies $\nq$.
Let $(c_p)_{p} \in \Lambda_{\M}$ be arbitrary. Choose $\psi \in \DM(0,1)$ such that $\mu_1^0(\psi) \neq 0$. Define $\chi = \widehat{\psi}_{|[0,\infty)} \in \CM$ and $G = \mathcal{L}(\chi)_{|\R}$. Notice that, by Lemma \ref{reduction}, we have that $G(0) = \mu_0(\chi) = i \mu_1^0(\psi) \neq 0$. Set
$$
b_p = (-i)^p \sum_{j = 0}^p \binom{p}{j} i^jc_j \left ( \frac{1}{G} \right)^{(p-j)} (0), \qquad p \in \N.
$$
We claim that $(b_p)_{p} \in \Lambda_{\M}$. Indeed, choose $C,h > 0$ such that $|c_p| \leq C h^p p! M_p$ for all $p \in \N$. Next, since $G \in \mathcal{E}^{\M}_\infty(\R)$ (Lemmas \ref{extension} and \ref{inj-lapl}), $G(0) \neq 0$ and $\M$ satisfies $\lc$, a classical result of Malliavin on the inverse-closedness of algebras of ultradifferentiable functions \cite[p.\ 185, 4.1]{Malliavin} implies that there are $C', k > 0$ such that $ |(1/G)^{(p)}(0)| \leq C' k^p p!M_p$ for all $p \in \N$. Hence,
$$
|b_p| \le C C'\sum_{j = 0}^p \binom{p}{j} h^j j!M_j k^{p-j} (p-j)!M_{p-j} \leq  CC' (h +k)^p p!M_p, \qquad p \in \N,
$$
where we have used~\eqref{eq-conseq-lc}. 
By Theorem \ref{surj-stiel-modified}, part $(v)\Rightarrow(iii)$, there is $\varphi \in  \CM$ such that $\mu_p(\varphi) = b_p$ for all $p \in \N$.  Set
$$
\widetilde{\varphi}(x) =
\left\{
	\begin{array}{ll}
		\mbox{$\varphi(x)$},  &  x \geq 0, \\ \\
		\mbox{$0$},  &  x < 0.
	\end{array}
\right.
$$
Then, $\widetilde{\varphi} \ast \widehat{\psi} \in \mathcal{F}(\DM(0,1))$ (cf.\ the proof of Proposition~\ref{inj-origin}). Finally, Lemma \ref{inversion-2} implies that
$$
\mu_p(\widetilde{\varphi} \ast \widehat{\psi}) =  \sum_{j=0}^p \binom{p}{j}  b_j \mu_{p-j}(\widehat{\psi}) =  \sum_{j=0}^p \binom{p}{j}  b_j \mu_{p-j}(\chi) =  c_p, \qquad p \in \N,
$$
and so $(ii)$ holds.
\end{proof}

\begin{corollary}\label{coro_notbijection_origin}
Let $\M$ be a weight sequence satisfying $\lc$ and $\dc$. Then,  $\mathcal{M}: \mathcal{F}(\DM(0,1)) \rightarrow \Lambda_{\M}$ and
$\mathcal{M}^0: \DM(0,1) \rightarrow \Lambda_{\widecheck\M}$ are never bijective.
\end{corollary}
\begin{proof}
In view of Theorems \ref{inj-origin} and \ref{surj-origin-modified}, this can be shown in exactly the same way as Corollary \ref{cor-bij}.
\end{proof}

\begin{remark}
A careful inspection of the proofs of Theorems~\ref{surj-stiel-modified} and~\ref{surj-origin-modified} shows that, as long as $\M$ is a weight sequence satisfying $\lc$ and $\dc$, the surjectivity of the Borel mapping $\mathcal{B}: \mathcal{A}_{\M}(\HH) \rightarrow \Lambda_{\M}$ implies, not only $\ga(\M) > 1$, but also the surjectivity of all the moment mappings considered in both statements. Although the condition  $\mg$, combined with $\lc$ and $\ga(\M) >1$, allows one to prove the surjectivity of $\mathcal{B}$ (see Theorem \ref{surj-borel}), in some cases where $\mg$ fails one can still show that $\mathcal{B}$ is surjective. A very classical example is that of the so-called $q$-Gevrey sequences, $\M_q=(q^{p^2})_{p\in\N}$, where $q>1$. These sequences satisfy $\lc$, $\dc$ and $\ga(\M_q) > 1$ (indeed, $\ga(\M_q)=\infty$) but not $\mg$. One can prove (see~\cite[Subsect.~3.3]{thilliez} for some hints and references) that $\mathcal{B}: \mathcal{A}_{\M_q}(\HH) \rightarrow \Lambda_{\M_q}$
is surjective and so the previous considerations apply to this case.
\end{remark}

\noindent\textbf{Acknowledgements}: The first author is supported by FWO-Vlaanderen, via the postdoctoral grant 12T0519N. The last two authors are partially supported by the Spanish Ministry of Economy, Industry and Competitiveness under the project MTM2016-77642-C2-1-P.\par
The authors wish to express their gratitude to Prof. Manuel N\'u\~nez, from the Universidad de Valladolid (Spain), for making them aware of Theorem~\ref{thPrivalov}.

\end{document}